\documentclass[reqno,12pt]{amsart}
\usepackage{amsmath,amsfonts,amssymb,amsthm}
\usepackage{etoolbox}
\usepackage{latexsym}
\usepackage{mathrsfs}

\usepackage[mathcal]{euscript}
\usepackage{amstext}
\usepackage{bm}
\usepackage{color}

\usepackage{cite}

\usepackage{esint}

\newcommand{\T}{\mathbb{T}}

\newcommand{\Z}{\mathbb{Z}}

\newcommand{\cB}{\mathcal{B}}
\newcommand{\cC}{\mathcal{C}}

\newcommand{\cF}{\mathcal{F}}

\newcommand{\cM}{\mathcal{M}}
\newcommand{\cN}{\mathcal{N}}

\renewcommand{\tilde}{\widetilde}
\renewcommand{\hat}{\widehat}
\renewcommand{\bar}{\overline}

\newcommand{\supp}{\mathrm{supp} \,}

\usepackage[pagebackref=true, colorlinks=true, citecolor=blue, pdfborder={0 0 .1},breaklinks]{hyperref}

\usepackage[initials,nobysame]{amsrefs}

\usepackage{enumitem}

\definecolor{magenta}{rgb}{1.0, 0.0, 1.0}
\definecolor{deepmagenta}{rgb}{0.8, 0.0, 0.8}
\definecolor{electricpurple}{rgb}{0.75, 0.0, 1.0}
\definecolor{lavenderindigo}{rgb}{0.58, 0.34, 0.92}
\definecolor{deeppink}{rgb}{1.0, 0.08, 0.58}
\definecolor{vividviolet}{rgb}{0.62, 0.0, 1.0}
\definecolor{darkorchid}{rgb}{0.6, 0.2, 0.8}
\definecolor{dartmouthgreen}{rgb}{0.05, 0.5, 0.06}
\definecolor{green(ncs)}{rgb}{0.0, 0.62, 0.42}
\definecolor{lawngreen}{rgb}{0.49, 0.99, 0.0}
\definecolor{green(pigment)}{rgb}{0.0, 0.65, 0.31}
\definecolor{green(html/cssgreen)}{rgb}{0.0, 0.5, 0.0}
\definecolor{neongreen}{rgb}{0.22, 0.88, 0.08}
\definecolor{orange(colorwheel)}{rgb}{1.0, 0.5, 0.0}
\definecolor{phlox}{rgb}{0.87, 0.0, 1.0}
\definecolor{tangelo}{rgb}{0.98, 0.3, 0.0}
\definecolor{yellow}{rgb}{1.0, 1.0, 0.0}
\definecolor{cadmiumyellow}{rgb}{1.0, 0.96, 0.0}
\definecolor{laserlemon}{rgb}{1.0, 1.0, 0.13}
\definecolor{uscgold}{rgb}{1.0, 0.8, 0.0}
\definecolor{uclagold}{rgb}{1.0, 0.7, 0.0}
\definecolor{darkspringgreen}{rgb}{0.09, 0.45, 0.27}
\definecolor{darkpastelgreen}{rgb}{0.01, 0.75, 0.24}
\definecolor{brightgreen}{rgb}{0.4, 1.0, 0.0}

\newtheorem{theorem}{Theorem}[section]

\newtheorem{remark}[theorem]{Remark}

\newtheorem{definition}[theorem]{Definition}
\newtheorem{lemma}[theorem]{Lemma}
\newtheorem{proposition}[theorem]{Proposition}
\newtheorem{corollary}[theorem]{Corollary}

\numberwithin{equation}{section}
\numberwithin{theorem}{section}

\newcommand{\vertiii}[1]{{\left\vert\kern-0.25ex\left\vert\kern-0.25ex\left\vert #1
    \right\vert\kern-0.25ex\right\vert\kern-0.25ex\right\vert}}

\newcommand{\dsp}{\displaystyle}

\renewcommand{\hat}{\widehat}
\renewcommand{\tilde}{\widetilde}

\newcommand{\real}{\mathbb{R}}
\newcommand{\dd}{\ \mathrm{d}}

\DeclareMathOperator{\dv}{div }
\DeclareMathOperator{\curl}{curl}

\title[Absence of anomalous dissipation]{Absence of anomalous dissipation for vortex sheets}

\author[Elgindi]{Tarek M. Elgindi$^1$}
\author[Lopes Filho]{Milton C. Lopes Filho$^2$}
\author[Nussenzveig Lopes]{Helena J. Nussenzveig Lopes$^2$}
\address{$^1$ Department of Mathematics, Duke University, Durham, NC 27708-0320 -- USA}
\email{tarek.elgindi@duke.edu}
\address{$^2$ Instituto de Matem\'atica, Universidade Federal do Rio de Janeiro, Cidade Universit\'aria -- Ilha do Fund\~ao, Caixa Postal 68530, 21941-909 Rio de Janeiro, RJ -- BRAZIL}
\email{mlopes@im.ufrj.br}
\email{hlopes@im.ufrj.br}

\begin{document}

\begin{abstract}   A family of solutions of the incompressible Navier-Stokes equations is said to present anomalous dissipation if energy dissipation due to viscosity does not vanish in the limit of small viscosity. In this article we present a proof of absence of anomalous dissipation for 2D flows on the torus, with an arbitrary non-negative measure plus an integrable function as initial vorticity and square-integrable initial velocity. Our result applies to flows with forcing and provides an explicit estimate for the dissipation at small viscosity. The proof relies on a new refinement of a classical inequality due to J. Nash.
\end{abstract}

\maketitle

{\bf Compiled on: {\today}}

\tableofcontents

\section{Introduction}


This article concerns families of solutions of the incompressible Navier-Stokes equations on the torus $\T^2 = [-\pi, \pi]^2$, given by:

\begin{equation} \label{eq:NS}
\left\{\begin{array}{ll}
\partial_t u^\nu + (u^\nu \cdot \nabla) u^\nu = -\nabla p^\nu + \nu \Delta u^\nu + F^\nu, \quad & \text{ in } \T^2 \times (0,T) \\
\dv u^\nu = 0, \quad & \text{ in } \T^2 \times [0,T) \\
u^\nu(\cdot,0) = u^\nu_0, \quad & \text{ in } \T^2,
\end{array}
\right.
\end{equation}
where $\nu>0$ is the viscosity, $u^\nu$ the flow velocity, $p^\nu$ is the pressure and $F^\nu = F^\nu(\cdot,t)$ is an external force. If $\nu =0$, system \eqref{eq:NS}  is called the Euler system. Energy methods yield the  identity:
\begin{equation} \label{eq:viscousenergy}
\frac{1}{2} \frac{d}{dt}  \int |u^\nu|^2 dx = -\nu \int |\omega^\nu|^2 dx + \int F^\nu \cdot u^\nu dx,
\end{equation}
where $\omega^\nu = \curl u^\nu$ is the vorticity of the flow. 
The term
\begin{equation} \label{eq:rateofdissipation}
\zeta^\nu \equiv \nu \int \int |\omega^\nu|^2 \dd x \dd t
\end{equation}
is called the energy dissipation term, or simply the dissipation term, in this context.

For initially smooth flows, it is easy to see that $\zeta^\nu$ vanishes as $\nu \rightarrow 0$, leading to solutions of the Euler system which satisfy an exact energy balance. However, this is not necessarily the case if the flow is initially nonsmooth. A sequence of solutions $\{u^\nu\}$ for which $\lim_{\nu \to 0} \zeta^\nu > 0$ is said to exhibit {\it anomalous dissipation}. The issue of presence or absence of anomalous dissipation in the vanishing viscosity limit is central to the study of turbulence, see for example \cite{Frisch1995} and references therein.

In \cite{JLLN2024}, F. Jin and co-authors proved that, under suitable assumptions on the forcings, strong convergence of the vanishing viscosity approximations is equivalent to the inviscid limit satisfying energy balance. Part of their proof involves showing that strong convergence implies absence of anomalous dissipation; the converse of which, however, is not necessarily true. Additionally, it is  proved in \cite{JLLN2024} that strong convergence of the vanishing viscosity approximation follows if initial vorticity is in a rearrangement-invariant space compactly embedded in $H^{-1}$, see also \cite{LNT2000}. The conceptual gap between initial vorticities in these spaces and initial vorticities whose singular parts are  nonnegative bounded Radon measures, called {\it vortex sheet initial data} is that the $L^2$-compactness of velocity approximations is lost. Consequently, a compensated compactness argument is called for to show that the weak limit of the vanishing viscosity approximation is a weak solution of the inviscid equations, see \cite{Delort1991,Majda1993,VecchiWu93}. In view of \cite[Theorem 2.7]{JLLN2024} these weak solutions can, indeed, have an energy defect. Still, it makes sense to ask whether the viscous approximations exhibit anomalous dissipation. This question was answered in the negative by L. De Rosa and J. Park in the nice recent work \cite{DeRosaPark2024} and, independently, in the present work.
 
Although the results obtained here are similar to those obtained in \cite{DeRosaPark2024}, there are a few differences. First, our work applies to flows with forcing, and it allows for more general ways in which initial data may be approximated. Also, our proof is quite different, involving direct quantitative estimates, and, as a result, provides a rate of vanishing of dissipation with respect to viscosity. The key ingredient of our work is a new refinement of Nash's inequality, given in Proposition \ref{prop:L1XiExist}. In contrast, De Rosa and Park use the classical Nash inequality \eqref{e:Nash} to obtain a key estimate, and a mollification which makes clever use of parabolic scaling. 

The remainder of this paper is organized as follows. In Section 2 we prove a convolution inequality and use it to prove an interpolation inequality. In Section 3 we derive a refinement of an inequality due to J. Nash, which, in turn, is a limiting case of the Gagliardo-Niremberg inequalities. In Section 4 we present an adaptation,  to bounded measures, of the inequality in the previous section, replacing the uniform integrability of $L^1$ functions with an estimate of the total mass of a measure in small disks. It is only in the final sections that we turn to solutions of the incompressible flow equations. In Section 5 we prove absence of anomalous dissipation, assuming suuitable uniform bounds on the Navier-Stokes solutions, as viscosity vanishes. In Section 6 we apply the result Section 5 to flows with initial vorticities whose singular part is a nonnegative measure and we obtain a rate for the vanishing of the dissipation. Finally, in Section 7, we remark on an extension of our results, describe an example which shows the necessity of some of the conditions we use to prove absence of anomalous dissipation and we present a few concluding remarks.

\section{Real analysis preliminaries}

In this section we will develop some real analysis preliminaries which will serve as a basis for our work.

Let $f$ and $g$ be smooth functions on $\mathbb{R}^d$ and consider the convolution
\[f\ast g = (f\ast g) (x) \equiv\int f(y)g(x-y) \dd y.\]
By Young's inequality, we have that
\begin{equation} \label{eq:Young}
\|f\ast g\|_{L^2} \leq \|f\|_{L^1}\|g\|_{L^2} .
\end{equation}

In what follows we will show that we can refine \eqref{eq:Young} if $g$ is supported in a ball $B_*$, replacing $\|f\|_{L^1}$ in \eqref{eq:Young} by something smaller (in particular, related to the integral of $|f|$ on balls of the same size as $B_*$).

\begin{lemma} \label{lem:convolineq}
Let $f\in L^1(\mathbb{R}^d)$ and $g\in L^2(\mathbb{R}^d)$ and assume that $\supp g$ is contained in a ball $B_*.$
Then,

\[\|f*g\|_{L^2}\leq \left(\sqrt{\|f\|_{L^1}\sup_{|B|=|B_*|}\|f\|_{L^1(B)}}\right)\|g\|_{L^2}\]

\end{lemma}

\begin{proof}

Let us begin by noting that
\begin{equation} \label{eq:Linfty2Linfty}
\|f*g\|_{L^\infty}\leq \left(\sup_{B=|B_*|}\|f\|_{L^1(B)}\right) \|g\|_{L^\infty}.
\end{equation}

From  Young's inequality we have, additionally,
\begin{equation} \label{eq:L12L1}
\|f*g\|_{L^1}\leq \|f\|_{L^1}\|g\|_{L^1}.
\end{equation}

We now claim that
\begin{equation} \label{eq:interpclaim}
\|f*g\|_{L^2}\leq \|f\|_{L^1}^{1/2}\left(\sup_{|B|=|B_*|}\|f\|_{L^1(B)}\right)^{1/2}\|g\|_{L^2}.
\end{equation}

To see this let $L^p_{B_*}$ denote the Banach space of $L^p$ functions that vanish identically outside of $B_*$.

Fix $f\in L^1(\mathbb{R}^d)$ and consider the linear operator $I$ given by
\[
\begin{array}{ccc} I: L^p_{B_*}(\real^d) & \to & L^p (\real^d)\\
 g  & \mapsto & f*g
\end{array}.
\]

Then, from \eqref{eq:Linfty2Linfty}, we have
\[\|I\|_{L^\infty_{B_*}\rightarrow L^\infty} \leq \sup_{|B|=|B_*|}\|f\|_{L^1(B)}.\]
Also, \eqref{eq:L12L1} gives
\[\|I\|_{L^1_{B_*}\rightarrow L^1} \leq \|f\|_{L^1}.\]
We use the Riesz-Thorin theorem, see \cite{Folland}, to conclude that
\[\|I\|_{L^2_{B_*}\rightarrow L^2} \leq  \|f\|_{L^1}^{1/2}\left(\sup_{|B|=|B_*|}\|f\|_{L^1(B)}\right)^{1/2} .\]
This establishes the claim and, thus, concludes the proof.

\end{proof}


\begin{lemma} \label{lem:Nvare_estimate}
    There exists a universal constant $C>0$ such that, if $f \in H^1(\T^2)$ with $\dsp{\int_{\T^2} f = 0}$, $0<\alpha < 1/2$, $0<\varepsilon <1$ and $\cN \in \real^+$, then:
    \begin{align} \label{eq:Nvare_estimate}
    \| f \|_{L^2}^2 \leq & C\left[ \cN^2 \| f\|_{L^1}
    \left(\sup_{z \in \T^2} \int_{\{|z-y| < \varepsilon^\alpha\}} |f(y)| \, \mathrm{d}y + \| f\|_{L^1} \frac{1}{\cN\varepsilon^\alpha}  \right)\right. \nonumber \\
    & \left.  + \frac{1}{\cN^2} \|\nabla f \|_{L^2}^2 \right] .
    \end{align}
\end{lemma}

\begin{proof}
     We begin with the Plancherel identity:
     \begin{equation} \label{eq:Plancherel}
     \|f\|_{L^2}^2 = \sum_{k\in \Z^2} |\hat{f}(k)|^2,
     \end{equation}
where $\hat{f}(k)$ is the $k$-th Fourier coefficient of $f$.

Consider the $\ell_\infty$ norm of $z=(z_1,z_2)$ on $\T^2$, $|z|_\infty = \max\{|z_1|,|z_2|\}.$

Using \eqref{eq:Plancherel}, we have
\begin{align} \label{eq:first_estimate}
   & \|f\|_{L^2}^2  = \sum_{|k|_\infty < \cN} |\hat{f}(k)|^2 + \sum_{|k| \geq \cN} |\hat{f}(k)|^2 \nonumber \\
   & \leq \sum_{|k|_\infty < \cN} |\hat{f}(k)|^2 + \frac{1}{\cN^2}\sum_{|k| \geq \cN} C|k|^2|\hat{f}(k)|^2 \leq \sum_{|k|_\infty < \cN} |\hat{f}(k)|^2  + \frac{1}{\cN^2} \|\nabla f \|_{L^2}^2. 
\end{align}

     We introduce the projection of $f$ onto the first $N\equiv \lceil \cN \rceil $ Fourier modes, with respect to the $\ell_\infty$-norm on $\T^2$:
     \begin{equation} \label{eq:PNf}
         P_N f = \left( \hat{f} \, 1_{|k|_\infty < N} \right)^{\vee}.
     \end{equation}

Therefore $\sum_{|k|_\infty < \cN} |\hat{f}(k)|^2 \leq \|P_Nf\|_{L^2}^2.$

Note that, since we are working with the $\ell_\infty$-norm on $\T^2$, it follows that
     \begin{equation} \label{eq:PNfconvol}
         P_N f (x) = \int_{\T^2} f(x-y) D_N(y) \, \mathrm{d}y,
     \end{equation}
where $D_N$ is the square Dirichlet kernel, $D_N(z_1,z_2)=d_N(z_1)d_N(z_2)$, with $d_N(z)=\sum_{-N\leq k \leq N} e^{ikz}$, the $1$-dimensional Dirichlet kernel. Hence
\begin{equation} \label{eq:Dirichlet2dim}
    D_N (z_1,z_2) = \frac{\sin[(2N+1)(z_1/2)]}{\sin(z_1/2)} \, \frac{\sin[(2N+1)(z_2/2)]}{\sin(z_2/2)},
\end{equation}
see \cite{Rudin1976}.

     With this notation we now write, for $\varepsilon < 1$:
     \begin{equation}
         P_N f = f \ast (D_N 1_{|z|_\infty < \varepsilon^\alpha}) + f \ast (D_N 1_{|z|_\infty \geq \varepsilon^\alpha}) \equiv P_N^I f + P_N^O f.
     \end{equation}

We estimate each of these terms separately.
For $P_N^I$ we will use the refined estimate for convolutions given in Lemma \ref{lem:convolineq} with $g = D_N 1_{|z|_\infty < \varepsilon^\alpha}$. We deduce that
\begin{equation} \label{eq:PNIf_estimate}
    \|P_N^I f\|_{L^2}^2 \leq  \| f\|_{L^1(\T^2)} \left(\sup_{z \in \T^2} \int_{\{|z-y|_\infty < \varepsilon^\alpha\}} |f(y)| \dd y\right)\|D_N\|_{L^2(\T^2)}^2.
\end{equation}

It is easy to see that
\begin{equation} \label{eq:dN_estimate}
    \|D_N\|_{L^2}^2 = \|d_N\|_{L^2}^4 = 4\pi N^2.
\end{equation}
Using \eqref{eq:dN_estimate} in \eqref{eq:PNIf_estimate}  yields
\begin{equation} \label{eq:PNIf_estimate_final}
    \|P_N^I f\|_{L^2}^2 \leq C N^2 \| f\|_{L^1(\T^2)} \sup_{z \in \T^2} \int_{\{|z-y|_\infty < \varepsilon^\alpha\}} |f(y)| \, \mathrm{d}y.
\end{equation}

 Next we estimate $P_N^O f$. We begin with the elementary observation that, if $0<\rho < \pi$, then
     \begin{equation} \label{e:intdN2}
     \int_{\rho}^\pi d_N^2(x)\,\mathrm{d}x \leq \frac{\pi^2}{\rho}.\end{equation}

Using \eqref{e:intdN2} and the formula for $D_N$ given in \eqref{eq:Dirichlet2dim} we find
\begin{equation} \label{eq:DNO_estimate}
   \|D_N 1_{|z|_\infty \geq \varepsilon^\alpha} \|_{L^2(\T^2)}^2 \leq 8 \|d_N\|_{L^2(\T^1)}^2 \int_{\varepsilon^\alpha}^\pi d_N^2(x) \, \mathrm{d}x \leq CN\frac{1}{\varepsilon^\alpha}.
\end{equation}
By Young's inequality we deduce
\begin{equation} \label{eq:PNOf_estimate_final}
 \|P_N^O f\|_{L^2}^2 \leq C \|f\|_{L^1}^2 N\frac{1}{\varepsilon^\alpha}.
\end{equation}

From \eqref{eq:PNIf_estimate_final} and \eqref{eq:PNOf_estimate_final} it follows that
\begin{equation} \label{eq:PNf_estimate}
    \|P_N f\|_{L^2}^2 \leq C N^2 \| f\|_{L^1} \left( \sup_{z \in \T^2} \int_{\{|z-y|_\infty < \varepsilon^\alpha\}} |f(y)| \, \mathrm{d}y + \| f\|_{L^1} \frac{1}{N\varepsilon^\alpha} \right).
\end{equation}
 Finally, putting together \eqref{eq:PNf_estimate} and \eqref{eq:first_estimate} yields the desired result.

\end{proof}

\section{A refinement of Nash's inequality}

Our main results depend on a refinement of an inequality due to J. Nash  which, in dimension $2$, corresponds to 
\begin{equation} \label{e:Nash} 
\left(\|f\|_{L^2(\real^2)}^2\right)^2 \leq \|f\|_{L^1(\real^2)}^2 \|\nabla f\|_{L^2(\real^2)}^2,
\end{equation}
for all $f \in H^1(\real^2) \cap L^1(\real^2)$, see \cites{Nash1958,BDS2020}. This is the content of the present Section.


Fix $K>0$. Let $\cF$ be a family of functions in $L^1$ such that

\begin{enumerate}[label={\bf (\Alph*)}, ref=
\textcolor{black}{\bf \Alph*},nosep]
    \item \label{i:A}$\|f\|_{L^1} \leq K$ for all $f\in \cF$, and
    \item[]
    \item \label{i:B} \[ \lim_{r \to 0^+} \; \sup_{f \in \cF} \, \sup_{z\in \T^2}\int_{\{|z-y| < r\}} |f(y)| \dd y \, = \, 0.\]
\end{enumerate}

Let 
\begin{equation} \label{eq:eta}
    \eta(r) = \max\left\{ \frac{1}{K}\sup_{f\in\cF}\sup_{z \in \T^2} \int_{\{|z-y| < r\}} |f(y)| \dd y, \,\frac{r}{\pi}\right\}, \quad \text{ for } 0\leq r\leq \pi.
\end{equation}

It is immediate that, if $0\leq r \leq \pi$  and $f\in \cF$, then:
\begin{align}
    0\leq \eta \leq 1; \qquad r \leq \pi \eta(r); \label{eqalign:etabds1} \\
    \sup_{z \in \T^2} \int_{\{|z-y| < r\}} |f(y)| \dd y \leq K \, \eta(r). \label{eqalign:etabds3}
\end{align}

\begin{proposition} \label{cor:L1concave}
Let $K>0$ and consider  $\cF \subset H^1(\T^2)$ a family of mean-free functions satisfying \ref{i:A} and \ref{i:B}. Then there exists a universal constant $C>0$ such that
\begin{equation} \label{eq:etaest}
    \|f\|_{L^2}^2 \leq C \, (K^2+1) \, \|\nabla f\|_{L^2} \sqrt{\eta \left( \|\nabla f\|_{L^2}^{-1/4}\right)}
\end{equation}
for all $f \in \cF$ such that $\|\nabla f\|_{L^2} > 1$.
\end{proposition}

\begin{proof}
    Let $f \in \cF$  and assume $\|\nabla f\|_{L^2} > 1$. Let $0< \varepsilon < 1$ and use $\alpha = 1/4$ and
    $\cN=1/\sqrt{\varepsilon}$   
  in Lemma \ref{lem:Nvare_estimate}. Then \eqref{eq:Nvare_estimate} becomes
        \begin{align} \label{eq:varestL1}
    \| f \|_{L^2}^2 \leq & C\Biggl(
    \frac{\|f\|_{L^1}}{\varepsilon}
    \sup_{z \in \T^2} \int_{\{|z-y| < \varepsilon^{1/4}\}} |f(y)| \, \mathrm{d}y
   \, +  \,
    \frac{\| f\|_{L^1}^2}{\varepsilon} \varepsilon^{1/4} + \, \,\varepsilon \|\nabla f \|_{L^2}^2 \Biggr)  .
    \end{align}

We wish to use $\varepsilon^\ast$ in \eqref{eq:varestL1}, where
\[\varepsilon^\ast \equiv \frac{1}{\|\nabla f\|_{L^2}} \sqrt{\eta\left(\|\nabla f\|_{L^2}^{-1/4}\right)},\]
with $\eta$ the function that was introduced in \eqref{eq:eta}. We note in passing that the hypothesis $\|\nabla f\|_{L^2} > 1$ implies that $\|\nabla f\|_{L^2}^{-1/4} < 1$, thus in the domain of $\eta$; since $0\leq \eta \leq 1$ (see \eqref{eqalign:etabds1}) we  conclude that  $0<\varepsilon^\ast < 1$.

Let $\varepsilon=\varepsilon^\ast$. Then, for the first term on the right-hand-side of \eqref{eq:varestL1}, we have

\begin{align} \label{eq:term1varestL1}
    & \frac{\|f\|_{L^1}}{\varepsilon^\ast} \sup_{z \in \T^2}   \int_{\{|z-y| < (\varepsilon^\ast)^{1/4}\}} |f(y)| \, \mathrm{d}y  \leq
        \frac{1}{\varepsilon^\ast} K^2 \,
        \eta \left((\varepsilon^\ast)^{1/4}\right) \nonumber \\
    & \leq  \frac{K^2\|\nabla f\|_{L^2}}{\sqrt{\eta\left(\|\nabla f\|_{L^2}^{-1/4}\right)}} \,
    \eta \left(\frac{1}{\|\nabla f\|_{L^2}^{1/4}}\right) = K^2\|\nabla f\|_{L^2} \sqrt{\eta\left(\|\nabla f\|_{L^2}^{-1/4}\right)},
    \end{align}
where we used \eqref{eqalign:etabds3} in the first estimate and \eqref{eqalign:etabds1} in the second one.

Next, we estimate the second term in \eqref{eq:varestL1}:
\begin{align} \label{eq:term2varestL1}
\frac{\|f\|_{L^1}^2}{\varepsilon^\ast} (\varepsilon^\ast)^{1/4}  \leq
\frac{K^2}{\varepsilon^\ast} \pi \eta ((\varepsilon^\ast)^{1/4}) 
 \leq \pi K^2 \|\nabla f\|_{L^2} \sqrt{\eta\left(\|\nabla f\|_{L^2}^{-1/4}\right)},
\end{align}
where we first used \eqref{eqalign:etabds1}, followed by the previous estimate, \eqref{eq:term1varestL1}.

Finally, for the last term of \eqref{eq:varestL1} we have:
\begin{equation} \label{eq:term3varestL1}
    \varepsilon^\ast \|\nabla f\|_{L^2}^2 = \|\nabla f\|_{L^2}\sqrt{\eta \left( \|\nabla f\|_{L^2}^{-1/4}\right)}.
\end{equation}

Substituting \eqref{eq:term1varestL1}, \eqref{eq:term2varestL1} and \eqref{eq:term3varestL1} in \eqref{eq:varestL1} yields the desired result.
\end{proof}


Let us consider an extension of the function $\eta$, defined in \eqref{eq:eta}, to all of $[0,+\infty)$, given by
\begin{equation} \label{eq:bareta}
\overline{\eta} = \overline{\eta}(r) =
\left\{
\begin{array}{ll}
\eta(r) & \text{ if } 0 \leq r \leq \pi, \\
1 & \text{ if } r > \pi.
\end{array}
\right.
\end{equation}

We will use $\overline{\eta}$ to prove our next result.

\begin{proposition} \label{prop:L1XiExist}
Let $K>0$ and let $\cF \subset H^1(\T^2)$ be a family of mean-free functions such that \ref{i:A} and \ref{i:B} hold.
Then there exists $\Upsilon \in \cC^1$, convex, increasing and superquadratic, such that
\[\Upsilon\left( \|f\|_{L^2}^2 \right) \leq \|\nabla f \|_{L^2}^2,\]
for all $f \in \cF$.
\end{proposition}

\begin{proof}
Let $f\in \cF$.
We have already established that, if $\|\nabla f\|_{L^2}>1$, then
 \[\|f\|_{L^2}^2 \leq C  \, \|\nabla f\|_{L^2}
 \sqrt{\eta \left( \|\nabla f\|_{L^2}^{-1/4}\right)},\]
for some constant $C>0$, where $C$ depends on $K$ and
additional universal constants. Using the Poincar\'e inequality and the definition of $\bar\eta,$ it is not difficult to deduce that



\begin{equation} \label{eq:L1etabarest}
  \|f\|_{L^2}^2 \leq C   \|\nabla f\|_{L^2} \,
  \sqrt{\overline{\eta}\left( \pi \|\nabla f\|_{L^2}^{-1/4}\right)}
\end{equation}
for all $f \in \cF$.
We note that the function $\overline{\eta}$ is continuous and non-decreasing. 

Now let 
\[\Phi(x) = \int_0^x C\sqrt{\overline{\eta}\left(\pi y^{-1/4}\right)} \dd y.\] It is not difficult to see that $\Phi$ is concave, increasing, and sublinear. Furthermore, since $\bar\eta$ is non-decreasing, we have that 
\[\Phi(x)\geq Cx\sqrt{\overline{\eta}\left(\pi x^{-1/4}\right)}.\] Consequently, we have from \eqref{eq:L1etabarest} that 
\[\|f\|_{L^2}^2 \leq  \Phi\Big(\|\nabla f\|_{L^2}\Big).\] Since $\Phi$ is $\cC^1$, increasing, concave and sublinear it follows that it has an inverse $\Phi^{-1}$ which is $\cC^1$, increasing, convex and superlinear. 
Hence,
\begin{equation} \label{eq:L1Xi}
    \Upsilon=\Upsilon(x) = (\Phi^{-1})^2(x)
\end{equation}
belongs to $\cC^1$, is increasing, convex and superquadratic. Clearly, the function $\Upsilon$ is such that $\Upsilon (\|f\|_{L^2}^2) \leq \|\nabla f \|_{L^2}^2$ for all $f \in \cF$, as desired. This concludes the proof.

\end{proof}

\section{An extension to bounded Radon measures}

The purpose of this section is to obtain a special version of Proposition \ref{prop:L1XiExist} valid for families of mean-free functions in $H^1(\T^2)$ which can be written as the sum of a nonnegative Radon measure in $H^{-1}$ and an $L^p(\T^2)$ function, for some $p>1$. We assume that this family is uniformly bounded in $H^{-1} +L^p$.

Let $\mu$ be a nonnegative measure in $H^{-1}(\T^2)$. We begin by recalling a basic estimate on the mass of $\mu$ in small discs, namely
\begin{equation} \label{eq:logrho_estimate}
 \int_{\{|z-y|  < \rho\}}  \mathrm{d}\mu(y)  \leq C \|\mu\|_{H^{-1}} |\log \rho|^{-1/2}.
\end{equation}
This estimate was established, in this form, in \cite{Schochet1995}, where the author used it to study limits of singular solutions of the 2D incompressible Euler equations. It is also at the heart of the proof of existence of weak solutions to the 2D incompressible Euler equations with vortex sheet initial data of a distinguished sign, see \cite{Delort1991}.

We begin with a proposition which is analogous to Proposition \ref{cor:L1concave}.


\begin{proposition} \label{cor:Delortconcave}
   Let $f \in H^1(\T^2)$ such that $\dsp{\int_{\T^2} f = 0}$. Assume that $f = \mu + w$, with $\mu \in \cB \cM\cap H^{-1}$, $\mu \geq 0$, and $w \in L^p(\T^2)$, for some $p>1$. Then there exists a universal constant $C>0$ such that,
   \begin{equation} \label{eq:logest}
    \| f \|_{L^2}^2 \leq C\left[\|f\|_{L^1}\left(\|\mu\|_{H^{-1}} + \|w\|_{L^p} + \|f\|_{L^1} \right) +1\right]\frac{\|\nabla f\|_{L^2}}{\sqrt[4]{\log \|\nabla f\|_{L^2}}},
\end{equation}
whenever $\|\nabla f\|_{L^2} > e^2$.
\end{proposition}

\begin{proof}
We begin by estimating $\dsp{\int_{|z-y| < \rho} |f| \, dz}$. We have, using \eqref{eq:logrho_estimate}:
\begin{align*}
    \int_{|z-y| < \rho} |f| \, dz & \leq \int_{|z-y| < \rho} \mu \, dz + \int_{|z-y| < \rho} |w| \, dz \\
    & \leq \|\mu\|_{H^{-1}} |\log \rho|^{-1/2} + C \|w\|_{L^p} \rho^{2p/p-1}.
\end{align*}

    Let $0<\varepsilon < 1$ and choose $\cN=  1 / \sqrt{\varepsilon}$  in Lemma \ref{lem:Nvare_estimate}. Then \eqref{eq:Nvare_estimate} becomes
\begin{align} \label{eq:vare_estimate1}
    \|f\|_{L^2}^2 \leq & C\Biggl[\frac{1}{\varepsilon}\|f\|_{L^1} \left(\|\mu\|_{H^{-1}}|\log \varepsilon|^{-1/2} + \|w\|_{L^p} \varepsilon^{2\alpha p/(p-1)} + \|f\|_{L^1}\varepsilon^{-\alpha + 1/2} \right) \nonumber \\
    &  + \varepsilon \| \nabla f \|_{L^2}^2\Biggr].
\end{align}

Recall we chose $\alpha < 1/2$, so the power of $\varepsilon$ in the third term is positive. Clearly, the power of $\varepsilon$ in the second term is also positive. Therefore, if $0<\varepsilon<1$, then the logarithmic term dominates the algebraic powers of $\varepsilon$.

In summary, it follows from \eqref{eq:vare_estimate1} and
that, for all $0<\varepsilon<1$,
\begin{equation} \label{eq:vare_estimate}
    \|f\|_{L^2}^2 \leq C\left[\frac{1}{\varepsilon}\|f\|_{L^1}\left( \|\mu\|_{H^{-1}}| + \|w\|_{L^p}+\|f\|_{L^1}\right)|\log \varepsilon|^{-1/2} + \varepsilon
    \| \nabla f \|_{L^2}^2\right].
\end{equation}

Assume that
\begin{equation} \label{eq:largegradf}
\| \nabla f \|_{L^2} > e^2
\end{equation}
and choose
\begin{equation} \label{eq:choicevare}
    \varepsilon^\ast = \frac{1}{\| \nabla f \|_{L^2} \sqrt[4]{\log \| \nabla f \|_{L^2}^2}}.
\end{equation}
Clearly $0<\varepsilon^\ast < 1$.

It is easy to see that, using $\varepsilon = \varepsilon^\ast$ in \eqref{eq:vare_estimate}, we obtain the desired result.

\end{proof}

Proposition \ref{cor:Delortconcave} allows us to adapt Proposition \ref{prop:L1XiExist} to bounded Radon measures. The main difference between our next result and that of Proposition \ref{prop:L1XiExist} is that, since we have a precise estimate on the behavior of our functions in small disks, this leads to more precise control on the relative behavior of $\|\nabla f\|_{L^2}^2$ and $\|f\|_{L^2}^2$.

\begin{proposition} \label{prop:DelortThetaExist}
Let $\cF \subset H^1(\T^2)$ be a family of mean-free functions. Assume that, if $f \in \cF$ then $f = \mu + w$, with $\mu \in \cB \cM \cap H^{-1}$, $\mu \geq 0$, and $w \in L^p(\T^2)$, for some $p>1$.  Let $K>0$ and assume further that, for all $f \in \cF$, there exists a decomposition $f=\mu + w$ as above such that $\|\mu\|_{\mathcal{BM}} +\|\mu\|_{H^{-1}}+\|w\|_{L^p} \leq K$.

Then there exists $\Phi \in \cC^1$, concave, increasing, sublinear, $\Phi(x) = \mathcal{O} \left( x |\log x|^{-1/4} \right)$ for large $x$, such $\Upsilon = (\Phi^{-1})^2$ satisfies
\[\Upsilon\left( \|f\|_{L^2}^2 \right) \leq \|\nabla f \|_{L^2}^2, \text{ for all } f \in \cF.\]

\end{proposition}

\begin{proof}

It follows from \eqref{eq:logest} that, whenever $\|\nabla f\|_{L^2} > e^2$, we have
\begin{equation} \label{eq:largenablawest}
    \|f\|_{L^2}^2 \leq C \frac{\|\nabla f\|_{L^2}}{\sqrt[4]{\log \|\nabla f\|_{L^2}}},
\end{equation}
where $C=C (K)$.

We introduce
\[ \Phi_1 = \Phi_1 (x) \equiv C \frac{x}{\sqrt[4]{\log x}}, \quad \text{ for } x > e^2.\]

It follows from an elementary calculation that, if $x>e^2$, then $ \Phi_1^\prime >0$ and $ \Phi_1^{\prime\prime} < 0$, so that $ \Phi_1 $ is increasing and concave in this interval. In addition,  $\lim_{x \to + \infty}  \Phi_1^\prime (x) = 0$.

We extend $\Phi_1$ to $[0, e^2]$ so that the extension, still called $\Phi_1$, belongs to $\cC^1$ and is concave and increasing on all of $[0,+\infty)$.

Let, also,
\[\Phi_2=\Phi_2(x) = C_Px^2,\]
with $C_P$ the constant from the Poincar\'e inequality for mean-free functions.

Choose $L>0$ such that $\Phi_2(e^2) = L \Phi_1(e^2)$ and let $\varphi(x) \equiv L\Phi_1(x) - \Phi_2(x)$. Then $\varphi(0) = 0=\varphi(e^2)$ and $\varphi$ is concave. It follows that $\varphi (x) \geq 0$ for $x \in [0,e^2]$, i.e. $\Phi_2(x) \leq L\Phi_1(x)$ on $[0,e^2]$.

Let $\Phi=\Phi(x) = (L+1)\Phi_1(x)$. Clearly $\Phi\in \cC^1$, $\Phi$ is concave, increasing, sublinear, $\Phi(x) = 
\mathcal{O} \left( x |\log x|^{-1/4} \right)$ for large $x$, as desired.

Next, in view of \eqref{eq:largenablawest}, the definition of $\Phi$, and using the Poincar\'e inequality for mean-free functions together with the inequality $\Phi_2 \leq \Phi$ on $[0,e^2]$,  we obtain
\[\|f\|_{L^2}^2 \leq  \Phi(\|\nabla f\|_{L^2}), \text{ for all } f \in \cF.\] 

It is immediate that the inverse function, $\Phi^{-1}$, belongs to $\cC^1$ and that it is increasing, convex and superlinear.

Set
\begin{equation} \label{eq:ThetaDef}
    \Upsilon =\Upsilon(x) = (\Phi^{-1})^2(x).
\end{equation}
Then  $\Upsilon$ is $\cC^1$, convex, increasing, superquadratic, and satisfies the desired estimate.

This concludes the proof.

    \end{proof}

\section{No anomalous dissipation}

In this section we finally turn our attention to the incompressible flow equations, namely, the Navier-Stokes system \eqref{eq:NS} and the Euler system below:

\begin{equation} \label{eq:Euler}
\left\{\begin{array}{ll}
\partial_t u + (u \cdot \nabla) u = -\nabla p + F, \quad & \text{ in } \T^2 \times (0,T) \\
\dv u = 0, \quad & \text{ in } \T^2 \times [0,T) \\
u(0,\cdot) = u_0, \quad & \text{ in } \T^2.
\end{array}
\right.
\end{equation}


We will apply Proposition \ref{prop:L1XiExist} to show that, under certain assumptions, there is no anomalous dissipation for a family of weak solutions of the 2D incompressible Navier-Stokes equations \eqref{eq:NS} whose weak limit is a weak solution of the Euler equations. 


Let us begin by recalling the definition of a weak solution of \eqref{eq:Euler}.
\begin{definition} \label{def:weaksol2DEuler}
  Fix $T>0$, let $u_0 \in L^2 (\T^2)$ be a divergence-free vector field. Let $F\in
  L^1 (0,T;L^2 (\T^2))$ and assume that $\dv F(\cdot,t) = 0$ a.e. $t\in (0,T)$. A vector field $u\in L^\infty (0,T;L^2 (\T^2))$ is said to be a \em{weak solution} of \eqref{eq:Euler}, if the following conditions are satisfied:
    \begin{enumerate}
      \item for all divergence-free test vector fields $\Phi\in C_c^\infty(\T^2 \times [0,T))$, we have
\begin{multline}
\int_0^T
\int_{\T^2} \left\{ u\cdot \partial_t \Phi
+ u \cdot D \Phi u \right\}  \dd x \dd t
+ \int_{\T^2}  u_0(x) \cdot \Phi(x,0) \dd x \\
= \int_0^T \int_{\T^2} F \cdot \Phi \, dx \, dt;
\end{multline}
\item $\dv  u(\cdot,t) = 0$ holds in the sense of distributions, a.e. $t\in [0,T]$.
  \end{enumerate}
\end{definition}

The problem of existence of weak solutions with vortex sheet initial data, i.e. $u_0 \in L^2$, $\dv u_0=0$,  $\curl u_0=\omega_0 \in \cB\cM$,  was first addressed in pioneering work by R. DiPerna and A. Majda,  see \cites{DiPernaMajda87a,DiPernaMajda87b,DiPernaMajda88}, and, in this generality, it remains an open problem.
If $\omega_0 \in L^1 \cap H^{-1}$ then the existence of a weak solution was established in \cite{VecchiWu93}. If the singular part of $\omega_0$ belongs to $\mathcal{BM} \cap H^{-1}$ and has a distinguished sign then existence of a weak solution with this initial data is due to J.-M. Delort, see \cite{Delort1991}. Uniqueness of weak solutions has not been proved for such irregular initial data and it is unlikely to hold. It is particularly worthy of note that nonuniqueness was recently proved in the setting $\omega  \in L^\infty_t L^p_x$, for some $p>1$ but close to $1$, see \cite{BCK2024}. 

It is natural to concentrate on those weak solutions which arise as vanishing viscosity limits of solutions of the Navier-Stokes equations. These are called physically realizable weak solutions, originally introduced in \cite{CLNS2016}*{Definition 3} for unforced flows, see also \cite{LMPP2021a}*{Definition 2.2}, and in \cite{LN2022}*{Definition 2.5} for flows with forcing.


Anomalous dissipation refers to the vanishing viscosity approximation of a physically realizable weak solution. Our first theorem applies to such a family of approximations, called physical realizations of a given physically realizable weak solution.

\begin{theorem} \label{t:main} Fix $T>0$. 
    Let $\{u_0^\nu\}_{\nu > 0}$ be a family of divergence-free vector fields in $L^2(\T^2)$. Let $u^\nu \in L^\infty(0,T; L^2(\T^2)) \cap L^2(0,T; H^1(\T^2))$ be the solution to $2D$ Navier-Stokes with viscosity $\nu$, with initial data $u_0^\nu$, and forcing $F^\nu \in L^2(0,T; L^2(\T^2))$. Set $\omega^\nu \equiv \curl u^\nu$. Assume the following conditions:

    \begin{enumerate}[label={{\rm {\bf H(\alph*)}}},
ref=\textcolor{black}{{\rm {\bf H(\alph*)}}},nosep]
        \item[]
        \item \label{i:Hi} $u_0^{\nu} \to u_0$ strongly in $L^2(\T^2)$;
        \item[]
        \item \label{i:Hii} $F^\nu \rightharpoonup F$ weakly in $L^2(0,T; L^2(\T^2))$;
        \item[]
        \item \label{i:Hiii} $\omega^\nu$ is bounded in $L^\infty(0,T; L^1(\T^2))$.
        \item[]
    \end{enumerate}

If additionally 
\begin{equation} \label{e:noDiracs}
   \lim_{r\to 0^+} \,\, \sup_{\nu >0} \, \sup_{t\in (0,T)} \, \sup_{z \in \T^2} \,\, \int_{\{|x-z|<r\} } |\omega^\nu| \, \mathrm{d}x = 0,
\end{equation}
    then, passing to subsequences as needed, $u^\nu$ converges, weak-$\ast$ $L^\infty(0,T; L^2(\T^2))$, to  a (physically realizable) weak solution of $2D$ Euler with initial data $u_0$ and 
    \begin{equation*}
        \limsup_{\nu \to 0^+} \nu \int_0^T \|\omega^\nu(\cdot, t)\|_{L^2}^2 \, \mathrm{d}t = 0.
    \end{equation*}
\end{theorem}

\begin{proof}
  
We begin by noting that, from hypotheses \ref{i:Hi}, \ref{i:Hii} and \ref{i:Hiii}, we have:
\[ \|u_0^\nu\|_{L^2} + \|\omega^\nu\|_{L^\infty_t L^1_x} + \|F^\nu\|_{L^2_tL^2_x} \leq K,\]
for some $K>0$ and for all $\nu>0$.
 
Recall the energy balance identity for $u^\nu$:
    \begin{equation} \label{eq:energid}
        \frac{1}{2}\|u^\nu(t)\|_{L^2}^2 = \frac{1}{2}\|u^\nu_0\|_{L^2}^2 - \nu \int_0^t \|\omega^\nu(\tau)\|_{L^2}^2 \dd \tau + \int_0^t \langle F^\nu, u^\nu \rangle_{L^2} \dd \tau ,
    \end{equation}
    
Therefore, from \eqref{eq:energid}, we find
\begin{equation} \label{eq:energest}
    \|u^\nu(\cdot,t)\|_{L^2} \leq C, \text{ for } t \in [0,T],
\end{equation}
for some $C=C(K,T)>0$, see also \cite{JLLN2024}*{Lemma 3.1}.

 It follows from \eqref{eq:energid} and \eqref{eq:energest} that the dissipation term is bounded:
 \begin{equation} \label{eq:bdddissip}
     \nu\int_0^t \|\omega^\nu(\cdot,\tau)\|_{L^2}^2 \dd \tau \leq C,
 \end{equation}
for some $C=C(K,T)>0$ and for all $t\in [0,T]$, see also \cite{JLLN2024}*{(3.6)}.

The evolution of $\omega^\nu$ is governed by the vorticity formulation of the Navier-Stokes equations \eqref{eq:NS} on the torus, given by
\begin{equation} \label{eq:vorticityNS}
    \left\{\begin{array}{ll}
\partial_t \omega^\nu + (u^\nu \cdot \nabla) \omega^\nu = \nu\Delta \omega^\nu + \curl F^\nu, \quad & \text{ in } \T^2 \times (0,T) \\
\dv u^\nu = 0, \, \curl u^\nu = \omega^\nu, \quad & \text{ in } \T^2 \times [0,T) \\
\omega^\nu(0,\cdot) = \omega^\nu_0, \quad & \text{ in } \T^2.
\end{array}
\right.
\end{equation}

Let $r>0$. Energy methods for the vorticity equation \eqref{eq:vorticityNS} provide the following estimate:
\begin{equation} \label{eq:enstrenest}
    \|\omega^\nu(\cdot,t)\|_{L^2}^2 \leq  \|\omega^\nu(\cdot,r)\|_{L^2}^2 \, -
    \nu \int_r^t \|\nabla \omega^\nu(\cdot,\tau)\|_{L^2}^2 \dd \tau \, + \, \frac{1}{\nu} \int_r^t \|F^\nu(\cdot,\tau)\|_{L^2}^2 \dd \tau,
\end{equation}
for any $t\geq r$, see for example \cite{JLLN2024}*{(3.14)}. It follows from \eqref{eq:enstrenest} together with \eqref{eq:bdddissip} that
\begin{equation} \label{eq:enstrest}
    \|\omega^\nu(\cdot,t)\|_{L^2}^2 \leq \frac{C}{t\nu},
\end{equation}
for all $t>0$, see also \cite{JLLN2024}*{Lemma 3.7}.

Define:
\begin{equation} \label{e:zetanu}
\zeta^\nu = \zeta^\nu(t) = \nu \int_0^t \|\omega^\nu(\tau)\|_{L^2}^2 \dd \tau.
\end{equation}

Let
\[\cF = \{\omega^\nu(\cdot,t), \;\; 0<t<T, \,\nu>0\}.\]
Note that, in view of \eqref{eq:enstrest}, $\cF \subset H^1(\T^2)$ and, furthermore, since $\omega^\nu = \curl u^\nu$, it follows that $\omega^\nu$ is mean-free. By hypothesis \ref{i:Hiii} and assumption \eqref{e:noDiracs} this family satisfies conditions \ref{i:A} and \ref{i:B}. Thus we can use Proposition \ref{prop:L1XiExist} to obtain a convex, increasing and superquadratic function $\Upsilon$ such that
\begin{equation} \label{eq:Upsilonenstrophy}
    \Upsilon(\|\omega^\nu(\cdot,t)\|_{L^2}^2) \leq \|\nabla\omega^\nu(\cdot,t)\|_{L^2}^2.
\end{equation}

Fix $\delta>0$ and introduce
\[\zeta^\nu_\delta (t) = \nu \int_\delta^t \|\omega^\nu(\cdot,\tau)\|_{L^2}^2 \dd \tau .\]

From \eqref{eq:enstrenest} and \eqref{eq:enstrest} we obtain
\begin{equation} \label{eq:enstrestNEW}
     \nu \|\omega^\nu(\cdot,t)\|_{L^2}^2 \leq  \frac{C}{\delta} -
    \nu^2 \int_\delta^t \|\nabla \omega^\nu(\cdot,\tau)\|_{L^2}^2 \dd \tau + K.
\end{equation}

Now, since the left-hand-side of \eqref{eq:enstrestNEW} is nonnegative, it follows that
\[\nu^2 \int_\delta^t \|\nabla \omega^\nu(\cdot,\tau)\|_{L^2}^2 \dd \tau \leq K + \frac{C}{\delta}.\]

Using \eqref{eq:Upsilonenstrophy} we find
\[\nu^2 \int_\delta^t \Upsilon(\|\omega^\nu(\cdot,\tau)\|_{L^2}^2) \dd \tau \leq K + \frac{C}{\delta}.\]

Recall $\Upsilon$ is convex. Then Jensen's inequality yields
\[\nu^2 (t-\delta) \Upsilon \left( \frac{1}{t-\delta} \int_\delta^t \|\omega^\nu(\cdot,\tau)\|_{L^2}^2  \dd \tau \right) \leq K + \frac{C}{\delta}.\]

Therefore we have
\begin{equation} \label{eq:UpsilonenstrophyNEW}
    \nu^2 (t-\delta) \Upsilon \left(  \frac{\zeta^\nu_\delta (t)}{\nu(t-\delta)} \right) \leq K + \frac{C}{\delta},
\end{equation}

which we re-write as
\begin{equation} \label{eq:L1Upsilonrate}
\Upsilon\left( \frac{\zeta^\nu_\delta}{\nu(t-\delta)}\right) \leq \frac{K+ C/\delta}{\nu^2(t-\delta)}.
\end{equation}

For convenience let $M= K+C/\delta$.

Recall that, by construction, $\Upsilon=(\Phi^{-1})^2,$ where $\Phi$ is a concave, increasing, sublinear function. Therefore \eqref{eq:L1Upsilonrate} yields
\[\Phi^{-1}\left( \frac{\zeta^\nu_\delta}{\nu(t-\delta)}\right) \leq \frac{\sqrt{M}}{\nu \sqrt{t-\delta}}.\]
Thus, since $\Phi$ is increasing, we obtain
\[\frac{\zeta^\nu_\delta(t)}{\nu(t-\delta)}  \leq \Phi \left(\frac{\sqrt{M}}{\nu \sqrt{t-\delta}}\right).\]
Hence
\[\zeta^\nu_\delta (t)\leq \nu (t-\delta)\,
\Phi \left(\frac{\sqrt{M}}{\nu \sqrt{t-\delta}}\right) \]
\[=\sqrt{M}\sqrt{t-\delta} \;  \frac{\Phi \left(\sqrt{M}/\nu \sqrt{t-\delta}\right)}{\sqrt{M}/\nu \sqrt{t-\delta}}.\]

Let $x_\nu = \dsp{\frac{\sqrt{M}}{\nu \sqrt{T-\delta}}}$. We have shown that
\begin{equation} \label{eq:L1rate}
    \zeta^\nu_\delta (T)\leq
\sqrt{M}\sqrt{T-\delta} \;\;  \frac{\Phi \left(x_\nu\right)}{x_\nu}.
\end{equation}

Since $x_\nu \to +\infty \text{ as } \nu \to 0^+,$
 it follows from the sublinearity of $\Phi$ that the right-hand-side above vanishes as $\nu \to 0^+$.

We have established, therefore, that, for every $\delta>0$,
\begin{equation} \label{eq:zetanudeltavan}
\limsup_{\nu \to 0^+} \zeta^\nu_\delta (T) = 0.
\end{equation}

To conclude the proof  it remains to examine
\[\nu \int_0^\delta \|\omega^\nu(\cdot,\tau) \|_{L^2}^2 \dd \tau.\]

Recall that, from \eqref{eq:energest}, $u^\nu$ is bounded in $L^\infty(0,T;L^2(\T^2))$. Passing to subsequences as needed it follows that $u^\nu$ converges  weakly-$\ast$ in $L^\infty(0,T;L^2(\T^2))$. That its weak-$\ast$ limit $u$ is a weak solution of the $2D$ Euler equations follows by a simple adaptation of the proof of \cite{Delort1991}*{Th\'eor\`eme 2.1.2}, see also \cite{Schochet1995}*{Theorem 3.3} and \cite{ILN2020}*{Proposition 2.1}, using the key assumption \eqref{e:noDiracs}.

Rewrite the energy balance equation \eqref{eq:energid}, this time from $0$ to $\delta$, as:
\[\nu \int_0^\delta \|\omega^\nu(\cdot,\tau)\|_{L^2}^2  \dd \tau = \frac{1}{2}\|u_0^\nu\|_{L^2}^2 - \frac{1}{2}\|u^\nu(\delta)\|_{L^2}^2 + \int_0^\delta \langle F^\nu , u^\nu \rangle_{L^2} \dd \tau.\]
We wish to take the $\limsup_{\nu \to 0^+}$ on both sides of this equation. To this end first recall that, since $u^\nu \rightharpoonup u$ weak-$\ast$ in $L^\infty(0,T;L^2(\T^2))$ we have, by weak lower semicontinuity, that
\[\|u(\delta)\|_{L^2}^2 \leq \liminf_{\nu \to 0^+} \|u^\nu (\delta)\|_{L^2}^2.\]
In addition,
\[\left|\int_0^\delta \langle F^\nu , u^\nu \rangle_{L^2} \dd \tau \right| \leq \|F^\nu\|_{L^2_tL^2_x}\sqrt{\delta}\|u^\nu\|_{L^\infty_tL^2_x}.\]
Therefore, recalling hypothesis \ref{i:Hi}, that $u_0^\nu \to u_0$ strongly in $L^2(\T^2)$, we obtain
\begin{equation} \label{eq:dissipestinitialtime}
\limsup_{\nu\to 0^+} \nu \int_0^\delta \|\omega^\nu(\tau)\|_{L^2}^2  \dd \tau \leq \frac{1}{2}\left(\|u_0\|_{L^2}^2 - \|u(\delta)\|_{L^2}^2\right) + C \sqrt{\delta}.
\end{equation}

Now, it was established in \cite{JLLN2024}*{Lemma 3.5} that, under the current hypotheses, the inviscid solution is right-continuous in $L^2$ at $t=0$, that is,
\[\lim_{t \to 0^+} \|u(t)-u_0\|_{L^2}^2 = 0,\]
Hence the right-hand-side of \eqref{eq:dissipestinitialtime} is vanishingly small as $\delta \to 0^+$.

Putting together \eqref{eq:zetanudeltavan} and \eqref{eq:dissipestinitialtime} and recalling the definition of $\zeta^\nu$, \eqref{e:zetanu}, we obtain
\begin{align*}
\limsup_{\nu \to 0^+} \zeta^\nu(T) & \leq \limsup_{\nu \to 0^+} \zeta^\nu_\delta (T) + \limsup_{\nu \to 0^+} \nu \int_0^\delta \|\omega^\nu(\tau)\|_{L^2}^2  \dd \tau \\
& \leq  \frac{1}{2}\left(\|u_0\|_{L^2}^2 - \|u(\delta)\|_{L^2}^2\right) + C \sqrt{\delta}.
\end{align*}
Above $\delta>0$ is arbitrary and the left-hand-side above vanishes as $\delta \to 0^+$. This establishes that
\[\limsup_{\nu \to 0^+} \zeta^\nu(T) = 0.\]

This concludes the proof.

\end{proof}

 \section{No anomalous dissipation for vortex sheets}

In this section we will use Theorem \ref{t:main} to show there is no anomalous dissipation for physical realizations of weak solutions with vorticities of vortex sheet regularity. We note that, while Theorem \ref{t:main} requires assumptions on the viscous solutions, namely \ref{i:Hiii} and \eqref{e:noDiracs}, for the results in the present section we place assumptions only on the data.

Fix $T>0$ and let $u^\nu \in L^\infty(0,T; L^2(\T^2)) \cap L^2(0,T; H^1(\T^2))$ be the solution to the $2D$ Navier-Stokes equations \eqref{eq:NS} with divergence-free initial data $u_0^\nu \in L^2(\T^2)$ and forcing $F^\nu \in L^2(0,T; L^2(\T^2))$. Assume hypotheses \ref{i:Hi} and \ref{i:Hii} hold true. Let $\omega_0^\nu = \curl u_0^\nu$.


\begin{theorem} \label{thm:mainL1}
Let $\omega^\nu\equiv \curl u^\nu$. Suppose that

\begin{enumerate}[label={{\rm {\bf H1(\alph*)}}},
ref=\textcolor{black}{{\rm {\bf H1(\alph*)}}},nosep]
    \item \label{i:H1i}
    $\{\omega_0^\nu\}_{\nu>0}$ is weakly compact in $L^1(\T^2)$;
    \item[]
    \item \label{i:H1ii} for almost every $\tau\in (0,T)$, $\{\curl F^\nu (\cdot,\tau)\}_{\nu>0}$ is weakly compact in $L^1(\T^2)$ and
    \[\int_0^T \sup_{\nu>0} \|\curl F^\nu(\cdot,\tau)\|_{L^1} \dd \tau < \infty.\]
    \end{enumerate}


Then we have:
\[\limsup_{\nu \to 0^+} \nu \int_0^T \|\omega^\nu (\tau)\|_{L^2}^2 \dd \tau = 0.\]
\end{theorem}

\begin{remark} \label{rem:L1choosephysrealiz}
    It is always possible to find solutions of Navier-Stokes $u^\nu$ with forcing $F^\nu$ satisfying conditions \ref{i:Hi}, \ref{i:Hii} as well as \ref{i:H1i} and \ref{i:H1ii} by, for instance, choosing $u_0^\nu$, $\omega_0^\nu$, $F^\nu$  to be  mollifications  of given divergence-free $u_0 \in L^2$, $\omega_0 \in L^1$ and appropriate $F$, respectively. Indeed, this is a consequence of an easy adaptation of the proof of  \cite{VecchiWu93}*{Theorem1}. 
\end{remark}

\begin{proof}


We wish to use Theorem \ref{t:main} so we need only check that hypotheses \ref{i:Hiii} and \eqref{e:noDiracs} hold true. 

We begin by noting that, from the assumptions \ref{i:Hi}, \ref{i:Hii}, \ref{i:H1i}, \ref{i:H1ii}, we have:
\[ \|u_0^\nu\|_{L^2} + \|F^\nu\|_{L^2_tL^2_x} + \|\omega_0^\nu\|_{L^1} + \|\curl F^\nu\|_{L^1_tL^1_x} \leq K,\]
for some $K>0$ and for all $\nu>0$.

Standard energy estimates give
\[\sup_{t\in (0,T)} \|\omega^\nu(\cdot,t)\|_{L^1} \leq \|\omega^\nu_0\|_{L^1} + \|\curl F^\nu\|_{L^1_tL^1_x} \leq K.\]

Thus hypothesis \ref{i:Hiii} is verified.

We now turn to \eqref{e:noDiracs}.

For $f \in L^1(\T^2)$ let us recall the definition of the rearrangement invariant maximal function $\cM_s(f)$:
\[\cM_s(f) \equiv \sup_{E\subset \T^2, |E|=s} \;\;\int_E |f(x)| \dd x.\]

The following estimate was established in \cite{JLLN2024}*{Proposition 4.5}:
\begin{equation} \label{eq:Msestimate}
\cM_s(\omega^\nu (t)) \leq \cM_s(\omega^\nu_0) + \int_0^T \cM_s(\curl F^\nu (\tau)) \dd \tau.
\end{equation}

We use \eqref{eq:Msestimate} to propagate the uniform integrability of $\omega^\nu_0$ and $\curl F^\nu(\tau)$, which follow from \ref{i:H1i} and \ref{i:H1ii}. We conclude that $\cM_s(\omega^\nu (t)) \to 0$ as $s\to 0$, uniformly in $t \in [0,T]$ and $\nu > 0$, and, consequently, \eqref{e:noDiracs} holds true.


The proof is complete.

\end{proof}

\begin{remark} \label{rem:L1rate}
We note that, in fact, our analysis yields a rate of vanishing, with respect to viscosity, of the dissipation, albeit not very explicit, given by \eqref{eq:L1rate}.

We will see, in our next result, that, in the case of physically realizable solutions for which the singular part of the initial vorticity is a bounded, nonnegative, Radon measure in $H^{-1}$, we obtain a very explicit rate.
\end{remark}


Next we apply Theorem \ref{t:main} to show absence of anomalous dissipation for physically realizations $u^\nu$ such that $u_0^\nu$ is compact in $L^2$ and $\omega_0^\nu$ is bounded in $\cB\cM_+ + L^p$, for some $p\geq 1$. We first address the case $p>1$.

\begin{theorem} \label{thm:mainDelort}
Fix $p>1$. Let $\omega^\nu \equiv \curl u^\nu$. Assume, in addition to \ref{i:Hi} and \ref{i:Hii}, that

\begin{enumerate}[label={{\rm {\bf H2(\alph*)}}}, ref=
\textcolor{black}{{\rm {\bf H2(\alph*)}}},nosep]
    \item \label{i:H2i}
    $\omega_0^\nu =\mu_0^\nu + w_0^\nu$, with $\{\mu_0^\nu\}_{\nu>0}$ bounded in $\cB\cM(\T^2)$, $\mu_0^\nu \geq 0$, and $\{w_0^\nu\}_{\nu>0}$ bounded in $L^p(\T^2)$;
    \item[]
    \item \label{i:H2ii} $\{\curl F^\nu\}_{\nu>0}$ bounded in $L^1(0,T;L^p(\T^2))$.
    \end{enumerate}


Then 
\[\limsup_{\nu \to 0^+} \nu \int_0^T \|\omega^\nu(\tau)\|_{L^2}^2 \dd \tau = 0.\]
\end{theorem}

\begin{remark} \label{rem:Delortchoosephysrealiz}
As in Remark \ref{rem:L1choosephysrealiz} we note that there always exists a family of Navier-Stokes solutions $u^\nu$ satisfying conditions \ref{i:Hi}, \ref{i:Hii}, \ref{i:H2i} and \ref{i:H2ii}. Again, it is enough to take $u_0^\nu$ and $F^\nu$ to be mollifications (in space) of $u_0\in L^2(\T^2)$, $\omega_0 = \mu_0 + w_0$, with $\mu_0 \in \mathcal{BM}_+$, $w_0 \in L^p$, and appropriate $F$. That such initial data and forcing leads to a physically realizable solution is an easy adaptation of \cite{Majda1993}, see also \cite{Delort1991} and \cite{Schochet1995}.
\end{remark}

\begin{proof}
    We start by noticing that the hypotheses imply
\[ \|u_0^\nu\|_{L^2} + \|F^\nu\|_{L^2_tL^2_x} + \|\mu_0^\nu\|_{\cB\cM} + \|w_0^\nu\|_{L^p}+ \|\curl F^\nu\|_{L^1_tL^p_x} \leq K,\]
for some $K>0$ and for all $\nu>0$. We also have, increasing the constant $K$ without renaming,  that
\[\|w_0^\nu\|_{L^1} + \|\curl F^\nu\|_{L^1_tL^1_x} \leq K.\]

As in Theorem \ref{t:main}, \eqref{eq:energest}, \eqref{eq:bdddissip}, \eqref{eq:enstrest}, we note that
\[\|u^\nu(\cdot,t)\|_{L^2} \leq C, \; \nu \int_0^t \|\omega^\nu(\cdot,\tau)\|_{L^2}^2 \dd \tau \leq C, \text{ and } \|\omega^\nu(\cdot,t)\|_{L^2}^2 \leq \frac{C}{t\nu},\]
for some $C=C(K,T)>0$ and for all $t>0$.

We now derive additional vorticity estimates.
We write the vorticity as
\[\omega^\nu = \mu^\nu + w^\nu,\]
where $\mu^\nu$ satisfies the equation
\begin{equation} \label{eq:muequation}
\partial_t \mu^\nu + u^\nu \cdot \nabla \mu^\nu = \nu \Delta \mu^\nu,    
\end{equation}
with $\mu^\nu(\cdot,0)=\mu^\nu_0$. Consequently $w^\nu$ is a solution of
\begin{equation} \label{eq:wequation}
\partial_t w^\nu + u^\nu \cdot \nabla w^\nu = \nu \Delta w^\nu + \curl F^\nu,    
\end{equation}
and $w^\nu(\cdot,0)=w^\nu_0$.

Standard estimates and the parabolic maximum principle now give
\[\|\mu^\nu(\cdot,t)\|_{L^1} \leq \|\mu^\nu_0\|_{\cB\cM},\]
\[\text{sign}[\mu^\nu(\cdot,t)]=\text{sign}[\mu^\nu_0], \]
and
\[\|w^\nu(\cdot,t)\|_{L^p}\leq \|w^\nu_0\|_{L^p} + \|\curl F^\nu\|_{L^1_tL^p_x}.\]
We also have
\[\|w^\nu(\cdot,t)\|_{L^1}\leq \|w^\nu_0\|_{L^1} + \|\curl F^\nu\|_{L^1_tL^1_x}.\]

Furthermore, since $\{u^\nu\}$ is bounded in $L^\infty(0,T;L^2(\T^2))$ it follows that
$\{\omega^\nu(\cdot,t)\}$ is bounded in $H^{-1}(\T^2)$. Recall that $L^q(\T^2)$ is continuously embedded into $H^{-1}(\T^2)$ for all $q>1$. Thus $w^{\nu}(\cdot,t)$ is bounded in $H^{-1}(\T^2)$. Therefore $\mu^{\nu}(\cdot,t)$ is also bounded in $H^{-1}(\T^2)$.
We have, hence, the following collection of estimates:
\begin{enumerate}
    \item $\|\omega^\nu(\cdot,t)\|_{L^1} \leq \|\mu^\nu(\cdot,t)\|_{L^1} + \|w^\nu(\cdot,t)\|_{L^1} \leq K$,
    \item $\|w^\nu(\cdot,t)\|_{L^p} \leq K$,
    \item $\|\mu^\nu(\cdot,t)\|_{H^{-1}}\leq \|\omega^\nu(\cdot,t)\|_{H^{-1}} + \|w^\nu(\cdot,t)\|_{H^{-1}} \leq K.$
\end{enumerate}

These bounds are enough to ensure that hypotheses \ref{i:Hiii} and \eqref{e:noDiracs} are valid. Moreover, we can use the convex, increasing, superquadratic function $\Upsilon$ obtained in Proposition \ref{prop:DelortThetaExist}, in the 
 corresponding part of the proof of Theorem \ref{t:main}.

This concludes the proof.

\end{proof}

\begin{remark} \label{rem:Delortrate}
We argued, in the proof of Theorem \ref{thm:mainDelort},  that we can use the function $\Upsilon$ obtained in Proposition \ref{prop:DelortThetaExist}. Recall that $\Upsilon$ is the square of the inverse of a function $\Phi$ which, near infinity, is a multiple of $x|\log x|^{-1/4}$. Therefore, under the hypotheses of Theorem \ref{thm:mainDelort}, we find that, for small $\nu >0$, \eqref{eq:L1rate} becomes
\begin{equation} \label{e:ratelognuminusquarter} 
\zeta^\nu_\delta (T) \lesssim |\log \nu|^{-1/4}, \;\; \text{ as } \nu \to 0^+.
\end{equation}
\end{remark}


In our final result we consider the case of vorticities in $\cB\cM_++L^1$.

\begin{corollary}
Let $\omega^\nu \equiv \curl u^\nu$. Assume, in addition to \ref{i:Hi} and \ref{i:Hii}, that

\begin{enumerate}[label={{\rm {\bf H3(\alph*)}}}, ref=
\textcolor{black}{{\rm {\bf H3(\alph*)}}},nosep]
    \item \label{i:H3i}
    $\omega_0^\nu =\mu_0^\nu + w_0^\nu$, with $\{\mu_0^\nu\}_{\nu>0}$ bounded in $\cB\cM(\T^2)$, $\mu_0^\nu \geq 0$, and $\{w_0^\nu\}_{\nu>0}$ bounded in $L^1(\T^2)$, where $\omega_0^\nu = \curl u_0^\nu$;
    \item[]
    \item \label{i:H3ii} $\{\curl F^\nu\}_{\nu>0}$ bounded in $L^1(0,T;L^1(\T^2))$.
    \end{enumerate}


Then 
\[\limsup_{\nu \to 0^+} \nu \int_0^T \|\omega^\nu(\tau)\|_{L^2}^2 \dd \tau = 0.\]
\end{corollary}

The proof is a standard combination of ideas from the proofs of Theorem \ref{thm:mainL1} and Theorem \ref{thm:mainDelort}, so we omit it.

\section{Comments and conclusions}

Hypothesis \eqref{e:noDiracs} means that no Dirac deltas appear in the vanishing viscosity limit of vorticity. This is a condition which has appeared many times in the literature and, in particular, it is central to J.-M. Delort's existence result for vortex sheet evolution with distinguished sign, see \cite{Delort1991} and see also \cite{Schochet1995}. A weaker, time-averaged, version of this hypothesis was used to extend Delort's result to the case of nonnegative mirror-symmetric vortex sheets, see \cite{LNX2001}. This raises the natural question as to whether no anomalous dissipation can be established under this time-averaged version of \eqref{e:noDiracs} as well. In fact, this issue was already addressed by De Rosa and Park in \cite{DeRosaPark2024} with an affirmative answer. We observe, below, that our version can also be carried out under this weaker hypothesis and we obtain rates as well. More precisely, let
 $0<t_1<t_2<T$. Consider a family $\cF \subset L^\infty(t_1,t_2;L^1(\T^2))$ and assume the following generalizations of \ref{i:A} and \ref{i:B} are valid:
\begin{enumerate}[label={\bf (\Alph*)}, ref=
\textcolor{black}{\bf \Alph*'},nosep]
    \item[\bf{(A')}] \label{i:A'} $\|f\|_{L^\infty_t L^1_x} \leq K$ for all $f\in \cF$, and
    \item[]
    \item[\bf{(B')}] \label{i:B'} \[ \lim_{r \to 0^+} \; \sup_{f \in \cF} \, \int_{t_1}^{t_2} \sup_{z\in \T^2}\int_{\{|z-y| < r\}} |f(y,t)| \dd y \dd t\, = \, 0.\]
\end{enumerate}
If, additionally, $\cF \subset L^2(t_1,t_2;H^1(\T^2))$ then it is possible to prove generalizations of Proposition \ref{cor:L1concave}, redefining appropriately the function $\eta$, and Proposition \ref{prop:L1XiExist}, so that $\|f\|_{L^2(\T^2)}$ and $\|\nabla f\|_{L^2(\T^2)}$ are substituted by $\|f\|_{L^2((t_1,t_2)\times \T^2)}$ and $\|\nabla f\|_{L^2((t_1,t_2)\times \T^2)}$, respectively. The function $\Upsilon$ in the statement of Proposition \ref{prop:L1XiExist} can be chosen independent of $t_1$, $t_2$.

Moreover, if {\bf(B')} is substituted by
\begin{enumerate}[label={\bf (\Alph*)}, ref=
\textcolor{black}{\bf \Alph*'},nosep]
    \item[\bf{(B'')}] \label{i:B''} \[ \sup_{f \in \cF} \, \int_{t_1}^{t_2} \sup_{z\in \T^2}\int_{\{|z-y| < r\}} |f(y,t)| \dd y \dd t\, \lesssim \, |\log r|^{-1/2}.\]
\end{enumerate}
then we can generalize Propositions \ref{cor:Delortconcave},  and \ref{prop:DelortThetaExist}. The function $\Upsilon$ in the statement of Proposition \ref{prop:DelortThetaExist} can be chosen independent of $t_1$, $t_2$ and follows the same construction as before.

Therefore, if we assume the weaker hypothesis {\bf(B')} for the family $\mathcal{F} = \{\omega^\nu\}_{\nu>0}$ then the conclusion of Theorem \ref{t:main} remains valid. If $\mathbf{(B^{\prime\prime})}$ holds for the singular part of vorticity, while the continuous part is bounded in $L^p$ for some $p>1$,  then we can obtain the same rate of dissipation as in \eqref{e:ratelognuminusquarter}.

The proof of the results claimed above is a straightforward adaptation of the work presented here. 

Next we wish to discuss a particular illustration of anomalous dissipation. For simplicity, we will present the example in the full plane.  Let $\varphi \in C^\infty_c(0,+\infty)$ and assume that
\[\int_0^{+\infty} s\varphi (s) \dd s = 0 \qquad \text{ and } \int_0^{+\infty} s|\varphi(s)| \dd s = \frac{1}{2\pi}.\]
Let $\omega_0=\omega_0(x) = \varphi(|x|)$, $x \in \real^2$. For each $\nu > 0$ set 
\[\omega_0^\nu = \omega_0^\nu (x) = \frac{1}{\nu^2} \omega_0 \left( \frac{x}{\nu}\right).\]
Then $\dsp{\int_{\real^2}} \omega_0^\nu (x) \dd x = 0$ and $\|\omega_0^\nu\|_{L^1(\real^2)}= 1$. 
Note that, if $\omega^\nu$ is the solution of the heat equation $\partial_t \omega^\nu = \nu \Delta \omega^\nu$ with initial data $\omega_0^\nu$ then $\omega^\nu$ is also the solution of the (unforced) Navier-Stokes equations with velocity $u^\nu (x,t) = \dsp{\frac{x^\perp}{2\pi |x|^2} \int_{B(0;|x|)} \omega^\nu (y,t) \dd y}$ since, in this case, the nonlinear term is curl-free. It is straightforward to show that:
\begin{enumerate}
    \item $\|\omega^\nu(\cdot,t)\|_{L^1} \leq 1$
    \item $\dsp{\nu \int_0^T \| \omega^\nu(\cdot,s)\|_{L^2}^2 \dd s } \geq C > 0$, for all $\nu > 0$.
\end{enumerate} 
Referring back to Theorem \ref{t:main} we note that \eqref{i:Hii} and \eqref{i:Hiii} hold true, but
\[|u^\nu_0|^2 \rightharpoonup C\delta_0 \text{ and } |\omega_0^\nu| \rightharpoonup \tilde{C} \delta_0, \text{ weakly in } L^1,\]
so that both \eqref{i:Hi}, strong convergence of the initial data, and \eqref{e:noDiracs}, the no-Diracs condition, are violated. It would be interesting to 
have examples that show that each one of these conditions are needed to prove absence of anomalous dissipation, specially the no-Diracs condition.

Finally, let us add a couple of final remarks and open problems. First, absence of anomalous dissipation does not imply strong convergence of the approximating sequence. This remains  an important open problem, equivalent to absence of {\it inviscid dissipation} for the limiting flow, as it was proved in \cite{JLLN2024}. Second, the rates of vanishing for the dissipation which we obtained here were not shown to be optimal.

\section*{Acknowledgments}  TME acknowledges support from a Simons Fellowship and the NSF DMS-2043024. MCLF was partially supported by CNPq, through grant \# 304990/2022-1, and by FAPERJ, through  grant \# E-26/201.209/2021. HJNL acknowledges the support of CNPq, through  grant \# 305309/2022-6, and that of FAPERJ, through  grant \# E-26/201.027/2022. This work was carried out while the authors were visiting AIM and the Department of Mathematics at Princeton  University. They thank both for their gracious hospitality. 
\bibliography{bibNoAnomDiss.bib}{}
\bibliographystyle{plain}

\end{document}